\documentclass[12pt]{amsart}

\RequirePackage[nonfrench]{kotex}
\usepackage[utf8]{inputenc}
\usepackage{kotex}
\usepackage{epsfig}

\usepackage{amsfonts}
\usepackage{amssymb}
\usepackage{amsmath}
\usepackage{amsthm}
\usepackage{amscd}
\usepackage{amstext}
\usepackage{latexsym}
\usepackage{array}
\usepackage{graphicx, subcaption}
\usepackage{todonotes}

\usepackage{enumerate}

\newcommand{\ZZ}{\mathbf{Z}}
\newcommand{\CC}{\mathbf{C}}
\newcommand{\QQ}{\mathbf{Q}}
\newcommand{\RR}{\mathbf{R}}

\newcommand{\DD}{\mathbf{D}}

\newcommand{\JJ}{\mathcal{J}}

\newcommand{\OO}{\mathcal{O}}

\newcommand{\al}{\alpha}

\DeclareMathOperator{\interior}{int}

\newcommand{\mfa}{\mathfrak{a}}

\newcommand{\mfab}{\mfa_\bullet}

\newcommand{\ep}{\epsilon}

\newcommand{\el}{\ell}

\newcommand{\qa}{\quad}
\newcommand{\vp}{\varphi}

\newcommand{\RRn}{\RR^n_{\ge 0}}

\newcommand{\EE}{\mathcal{E}}

\newcommand{\noi}{\noindent}

\providecommand{\abs}[1]{ |#1|}

\theoremstyle{plain}

\newtheorem{theorem}{Theorem}[section]

\newtheorem{lemma}[theorem]{Lemma}
\newtheorem{corollary}[theorem]{Corollary}
\newtheorem{proposition}[theorem]{Proposition}
\newtheorem{definition}[theorem]{Definition}

\newtheorem{question}[theorem]{Question}

 \newtheorem{example}[theorem]{\textnormal{\textbf{Example}}}

\newtheorem{Standard Process}[theorem]{Procedures}

\theoremstyle{remark}
\newtheorem{remark1}[theorem]{Remark}

\DeclareMathOperator{\Jump}{Jump}

\DeclareMathOperator{\jump}{Jump}

\begin{document}

\title[Jumping numbers of analytic multiplier ideals]{ Jumping numbers of analytic multiplier ideals \\ \small (with an appendix by S\'ebastien Boucksom)}

\author{ Dano Kim and Hoseob Seo }

\date{}

\maketitle

\begin{abstract}

\noindent  We extend the study of jumping numbers of multiplier ideals due to Ein-Lazarsfeld-Smith-Varolin from the algebraic case to the case of general plurisubharmonic functions. While  many properties from Ein-Lazarsfeld-Smith-Varolin are shown to generalize to the plurisubharmonic case, important properties such as periodicity and discreteness do not hold any more. Previously only two particular examples with a cluster point (i.e. failure of discreteness) of jumping numbers were known, due to Guan-Li and to Ein-Lazarsfeld-Smith-Varolin respectively.  We generalize them to all toric plurisubharmonic functions in dimension 2 by characterizing precisely when cluster points of jumping numbers exist and  by computing all those cluster points.  This characterization suggests that clustering of jumping numbers is a rather frequent phenomenon.  In particular, we obtain uncountably many new such examples.

\end{abstract}

\tableofcontents

\section{Introduction}

 Multiplier ideals are important tools in complex algebraic geometry (see e.g. \cite{L}, \cite{D11}, \cite{S98}). In \cite{ELSV}, Ein, Lazarsfeld, Smith and Varolin systematically studied jumping numbers of multiplier ideals in the algebraic case, i.e. for multiplier ideals $\JJ(cD)$ associated to an effective divisor $D$ or to an ideal sheaf on a smooth complex variety. A jumping number is a coefficient $c > 0$ at which the ideal $\JJ(cD)$ jumps down to a smaller ideal as the coefficient increases.

 Jumping numbers are fundamental invariants of the given singularity : they encode interesting geometric, algebraic and analytic information. When one considers all the nonreduced subschemes defined by $\JJ(cD)$ as $c$ varies, properties of jumping numbers are important as in recent work of Demailly~\cite{D15} on $L^2$ extension theorems from nonreduced subschemes. Also see the introduction and the references in \cite{ELSV} for earlier works on jumping numbers and several different connections and contexts in mathematics where the jumping numbers naturally arise.

    The full strength of the theory of multiplier ideals lies in that one can even take multiplier ideals from a \emph{plurisubharmonic} function, which can be regarded as a limit of divisors and ideals.  In \cite{ELSV},  the study of jumping numbers could not be extended to general plurisubharmonic functions since the openness conjecture was not yet proved at the time.

 Now that the openness conjecture is a theorem (see (\ref{openness})) of Guan and Zhou \cite{GZ} (after related works of several authors including \cite{FJ05} in dimension $2$ : see the references in \cite{GZ}), one can study the jumping numbers associated to general plurisubharmonic functions: general in the sense that it is not necessarily having the above mentioned algebraic singularities.

 More precisely, let $X$ be a complex manifold and let $\vp$ be a plurisubharmonic (\emph{psh}, for short) function on $X$. Let $x \in X$ be a point. From the openness theorem (\ref{openness}) of \cite{GZ}, it follows that (the stalks at $x$ of) the multiplier ideal sheaves $\JJ(c \vp)_x$ are \emph{constant precisely on} half-open intervals $[\chi, \chi' )$ in the sense that $\JJ(c \vp)_x = \JJ(\chi \vp)_x$ for $ c \in [\chi , \chi')$ and $ \JJ(\chi' \vp)_x \subsetneq \JJ(\chi \vp)_x .$  The real numbers such as $\chi$ and  $\chi'$ are called the jumping numbers of $\vp$ at $x$ since they are the points where the multiplier ideals ``jump".  This is parallel to the algebraic case as in \cite{ELSV}.

 However, the singularity of a general psh function can be highly complicated and difficult to study compared to the algebraic case, one reason being that one cannot apply a finite number of blow-ups  to resolve its singularity. Plurisubharmonic singularity is the topic of intensive current research (see e.g. \cite{D93}, \cite{D11}, \cite{DH}, \cite{DGZ}, \cite{R13a}, \cite{BFJ}). 
 
  One often expects peculiar behaviors coming from psh singularities which are not seen in the algebraic case.  Indeed for jumping numbers, recently Guan and Li \cite{GL} showed by an example (see \eqref{glpsh}) of a psh function  that the set of all the jumping numbers   can have a cluster point (i.e. limit point) in the psh case.  This is a new phenomenon in the general psh case.

 In this paper, along with giving many properties and examples of jumping numbers in the psh case, we first give another psh example (\ref{5.10}) having a cluster point of jumping numbers which is converted from an example of a graded system of ideals in \cite{ELSV}. For this conversion, we define a psh function associated to the graded system of ideals and use an important result Theorem~\ref{asymptotic} of S. Boucksom whose proof is given in the appendix.  These two examples \eqref{glpsh} and (\ref{5.10}) happen to be toric psh in dimension $2$.

 In view of the example (\ref{5.10}) where all positive integers are cluster points of the jumping numbers, it is natural to raise the following

\begin{question}\label{clusters}

 Let $\vp$ be a psh function on a complex manifold $X$ and let $x \in X$ be a point. Let $T$ be the set of cluster points in the set of jumping numbers $\jump(\vp)_x \subset \RR_{+}$. Assume that $T$ is nonempty. Is $T$ infinite?  Is $T$ unbounded?

\end{question}

 We answer this question positively when $\vp$ is toric psh: in that case, $T$ is indeed an infinite and unbounded set of positive real numbers by Corollary~\ref{mcmc}.

  On the other hand, we completely generalize the above two particular examples (which were all that were known) \cite{ELSV} (\ref{5.10}) and \cite{GL} \eqref{glpsh}  as follows.

\begin{theorem}[see   Theorem~\ref{cluster_asymptote}  for the full precise statement]\label{dimtwo}

Let $\vp$ be a toric psh function on the unit polydisk $\DD^2 \subset \CC^2$. Let $P(\vp)$ be the Newton convex body of $\vp$.  Let $\jump(\vp)_0$ be the set of jumping numbers of $\vp$ at the origin $0 \in \DD^2$. We have a complete characterization in terms of $P(\vp)$ for

(1)  when there exists a cluster point of jumping numbers in $\jump(\vp)_0$, and

(2) what are those cluster points.

\noi In particular, we obtain uncountably many new examples of psh functions having a cluster point of jumping numbers.

\end{theorem}

 This characterization (which is in the statement of Theorem~\ref{cluster_asymptote}) is explained very well by Figure 2 in \S 5.  One can then easily produce so many examples of psh functions with cluster points of jumping numbers, simply by taking Newton convex bodies with the behavior (d) or (h) of Figure 2 and then taking corresponding psh functions as in e.g. \eqref{siu2} in Example~\ref{graded}. This certainly suggests that clustering of jumping numbers is a rather frequent phenomenon in the singularity of psh functions.

 It is worth noting here that for the new examples in Theorem~\ref{dimtwo}, while we can  compute the cluster points of jumping numbers, it will be more difficult or impossible to compute all the jumping numbers themselves. This is unlike the particular case of (\ref{5.10}) where \cite{ELSV} computed all the jumping numbers first and then determined the cluster points from them.

\begin{remark1}
 The new examples in Theorem~\ref{dimtwo} also yield  new algebraic examples in the above graded system of ideals version of \cite{ELSV} since we can take the graded system of monomial ideals from the Newton convex bodies as in Example~\ref{graded}.
\end{remark1}

  Theorem~\ref{dimtwo} also provides further information on the above example \eqref{glpsh} of Guan and Li~\cite{GL} where $1$ was shown to be a cluster point of jumping numbers : now we show that all the cluster points are precisely the positive integers in Corollary~\ref{gl1}.

\qa

\bigskip

 Now turning to Theorem~\ref{asymptotic} of S. Boucksom in the appendix, we remark that it answers a natural question on asymptotic multiplier ideals, which may be of independent interest in algebraic geometry.

\begin{theorem}[= Theorem~\ref{asymptotic}]\label{asymptotic0}

Let $X$ be a smooth irreducible complex variety. Let $\mfab$ be a graded system of ideal sheaves on $X$. Let $\vp_{\mfab} := \log \left( \sum^{\infty}_{k \ge 1}  \ep_k  \abs{\mfa_k}^{\frac{1}{k}} \right) $ be a Siu psh function associated to $\mfab$  with a choice of coefficients $\ep_k > 0$ (see \eqref{siu}).  Then for every real $c > 0$, the analytic multiplier ideal sheaf $\JJ(c \vp_{\mfab})$ of the psh function $c \vp_{\mfab}$ is equal to the asymptotic multiplier ideal sheaf $\JJ  (c \cdot \mfab)$.

\end{theorem}

 Asymptotic multiplier ideals (see e.g. \cite{L}, \cite{Ka99}, \cite{ELS01}) were first brought into focus when they replaced the role of Siu psh functions \eqref{siu} in casting the crucial analytic proof of invariance of plurigenera by \cite{S98} in algebro-geometric language. In \cite{S98}, Siu psh functions appeared as local weight functions of a global hermitian metric, a Siu-type metric, of a line bundle $L$. This metric can be regarded as the metric with minimal singularities coming from taking all the (multi-)sections of $L$, hence it is a metric analogue of the stable base locus of $L$. (See e.g. \cite[Sec.4]{K15} for its comparison with Demailly's version of the metric with minimal singularities of $L$.)

  However there remained a natural question asking whether the algebraically defined asymptotic multiplier ideals and the analytic multiplier ideals of Siu psh functions are precisely equal to each other (not only playing analogous roles in the above proof). This is answered positively by Theorem~\ref{asymptotic} in the most general setting of the question, i.e. for a graded system of ideals (\cite{ELS01}, \cite{L}). Previously a special case was known by \cite{DEL} (see Remark~\ref{del}).

 The proof of Theorem~\ref{asymptotic} uses valuative characterizations of multiplier ideals \cite{BFJ}, \cite{BFFU}. The valuative approach to study psh singularities (see \cite{FJ05}, \cite{BFJ} and the references therein) plays an increasingly important role recently and it certainly influences our study in this paper to a great extent.

  In this valuative spirit, we define two psh functions $\vp$ and $\psi$ to be v-equivalent (in Sec. 2.1, cf. \cite[\S2]{K19}) if all their multiplier ideals are equal i.e. $\JJ(m \vp) = \JJ(m \psi)$ for all real $m > 0$. Thus two v-equivalent psh functions have all the same jumping numbers. But there exist many v-equivalent psh functions $\vp$ and $\psi$ that look quite different from each other (in particular $\vp - \psi$ is not locally bounded) as we show in Example~\ref{graded} and Proposition~\ref{toric_convex}.

 Finally, in addition to the absense of cluster points, we show that another important property of jumping numbers in the algebraic case shown in \cite{ELSV} fails in the psh case : periodicity. We use an interesting example of Koike~\cite{Ko15}.

This article is organized as follows. In Section 2, we define and study Siu psh functions associated to a graded system of ideals, which play an important role in this paper.  In Section 3, we define jumping numbers, present their basic properties and examples, generalizing counterparts from \cite{ELSV}. We also study more general ``mixed" jumping numbers. In Section 4, we show that periodicity of jumping numbers as in \cite{ELSV} fails for general psh functions, elaborating on an example of Koike \cite{Ko15}. In Section 5, we present and prove the main results on the cluster points of jumping numbers of psh functions.

% In this paper, we will generalize \cite{ELSV} to the psh setting as much as currently possible.

\begin{remark1}

 In a different context of algebraic multiplier ideals on singular varieties, there is an open question asking whether cluster points of jumping numbers ever exist in that algebraic setting (see e.g. \cite[Question 1.3]{G16}, \cite{U12}, \cite{T10}). See also \cite{BSTZ}.

\end{remark1}

\noi \textbf{Acknowledgements.}

 We are very grateful to S\'ebastien Boucksom for allowing us to include his proof of Theorem~\ref{asymptotic} in the appendix. We would like to thank Qi'an Guan for valuable communications and Alexander Rashkovskii for helpful comments on Lemma~\ref{Minkowski}. 

D.K. was partially supported by SRC-GAIA through  National Research Foundation of Korea grant No.2011-0030795. D.K. and H.S. were partially supported by Basic Science Research Program through  National Research Foundation of Korea funded by the Ministry of Education(2018R1D1A1B07049683) and by Seoul National University grant No.0450-20180027 (Mirae Gicho).

 \section{Multiplier ideals and plurisubharmonic functions }

For excellent general expositions on multiplier ideal sheaves and plurisubharmonic (psh) functions, we refer to \cite{D11} and also to \cite{L} for the algebraic theory of multiplier ideals.

\subsection{Valuative equivalence of psh functions}

 For a psh function $\vp$ defined on $U \subset X$, an open subset of a complex manifold, its multiplier ideal sheaf $\JJ(\vp)$ is the coherent ideal sheaf on $U$ consisting of holomorphic function germs $u$ such that $\abs{u}^2 e^{-2\vp}$ locally integrable with respect to Lebesgue measure.

 We have the usual notion of two psh functions $\vp, \psi$ having equivalent singularities \cite{D11}, i.e. when $\vp - \psi$ is locally bounded. We will often simply say $\vp$ and $\psi$ are equivalent.

   As a very convenient and flexible weaker notion (cf. \cite[\S2]{K19}),  we say that two psh functions $\vp$ and $\psi$ are \textbf{v-equivalent} and write $ \vp \sim_{v} \psi $ if the following equivalent conditions  hold:

 (1) For all $m > 0$, the multiplier ideals are equal : $\JJ(m\vp) = \JJ(m\psi)$.

 (2) At every point of all proper modifications over $X$, the Lelong numbers of $\vp$ and $\psi$ coincide. In other words, for every divisorial valuation $v$ centered on $X$, we have $v(\vp) = v(\psi)$.

% For example, a psh function $\vp$ with zero Lelong numbers everywhere is v-equivalent to constant function $\psi = 0$, but does not have equivalent singularities with $\psi$ if $\vp$ is not locally bounded.

 The above equivalence of (1) and (2) is due to \cite{BFJ} together with the following theorem in \cite{GZ}.

\begin{theorem}[Qi'an Guan and Xiangyu Zhou, Openness Theorem]\cite{GZ}\label{openness}

 Let $\psi$ and $\vp$ be two psh functions on a complex manifold $X$. We have the following equality of ideal sheaves

  $$ \bigcup_{\ep > 0} \JJ(\psi + \ep \vp) = \JJ(\psi). $$

\end{theorem}

 We remark that previously some special cases of this theorem (assuming $\psi = c \vp$ for a constant $c>0$)  were known: when $\dim X = 2$ in \cite{FJ05}, when $\vp$ is toric psh in \cite[(1.15)]{G} and  when $\vp$ is a Siu psh function (defined below) in \cite[(5.4)]{ELSV} together with using Theorem~\ref{asymptotic}.

%\subsection{Valuative characterization of multiplier ideal}

\subsection{Graded systems of ideals and  Siu psh functions}

  Here we show that the algebraic asymptotic multiplier ideal is equal to the analytic multiplier ideal of a Siu psh function associated to the graded system of ideals.

  Let $X$ be a smooth irreducible complex variety (or a complex manifold). Let $\mfab$ be a graded system of ideal sheaves on $X$ \cite[Section 2.4.B]{L}.  We define a Siu psh function associated to $\mfab$ following \cite{S98} (see also \cite[p.258]{BEGZ}, \cite{K15}) by

 \begin{equation}\label{siu}
   \vp_{\mfab} = \log \left( \sum^{\infty}_{k \ge 1}  \ep_k  \abs{\mfa_k}^{\frac{1}{k}} \right)  = \log \left( \sum^{\infty}_{k \ge 1}  \ep_k   \left( \abs{g_1^{(k)}} + \ldots + \abs{g_m^{(k)}}  \right)^{\frac{1}{k}} \right)
 \end{equation}

\noi on a domain $U \subset X$ where every graded piece $\mfa_k$ is an ideal with a choice of a finite number of generators, say $g_1^{(k)}, \ldots, g_m^{(k)}$ ($m$ is not fixed). It is convenient to use  the notation $ \abs{\mfa_k} := \abs{g_1^{(k)}} + \ldots + \abs{g_m^{(k)}}$ as in \eqref{siu} with the convention that each time the notation  $ \abs{\mfa_k}$ is used, a specific choice of a finite number of generators is made.

 Also $ \ep_k $'s are a choice of positive coefficients such that the infinite series converges. It is known that in general the singularity equivalence class of $\vp_{\mfab}$ depends on the choice of coefficients (see \cite{K15}).

 In algebraic geometry, the algebraic construction of asymptotic multiplier ideals associated to $\mfab$ is now standard and has been very useful (see  \cite[Chap.11]{L}). Following \cite[(11.1.15)]{L}, the asymptotic multiplier ideal sheaf of $\mfab$ with coefficient $c$, $\JJ(c \cdot \mfab)$ is defined to be the unique maximal member in the family of ideals $\{ \JJ (\frac{c}{q} \cdot \mfa_q) \}_{q \ge 1}$.

\begin{theorem}[S\'ebastien Boucksom]\label{asymptotic}

For a Siu psh function $\vp_{\mfab}$ defined above (on $U \subset X$), we have $\JJ(c\vp_{\mfab}) = \JJ  (c \cdot \mfab)$ for every real $c > 0$.

 \end{theorem}

\noi The proof of this theorem is in the appendix.  It is interesting to note that both Siu psh functions and asymptotic multiplier ideals first appeared in the context of the celebrated work by Yum-Tong Siu on the invariance of plurigenera \cite{S98} in the general type case and  in the subsequent algebraic translation \cite{Ka99}, \cite[(11.1.A)]{L} (and then also in \cite{ELS01}).   We immediately obtain the following

 \begin{corollary}\label{vequiv}

 Let $\vp =  \log \left( \sum^{\infty}_{k \ge 1}  \ep_k  \abs{\mfa_k}^{\frac{1}{k}} \right)  $ and $\tilde{\vp} =  \log \left( \sum^{\infty}_{k \ge 1}  \tilde{\ep}_k  \abs{\mfa_k}^{\frac{1}{k}} \right)  $ be two Siu psh functions associated to a graded system of ideals $\mfab$ with all $\ep_k , \tilde{\ep}_k > 0$. Then for every $c > 0$, we have $\JJ (c \vp ) = \JJ (c \tilde{\vp}) $.  In particular, $\vp$ and $\tilde{\vp}$ are v-equivalent psh functions.

 \end{corollary}

  Here we remark that for each $\mfa_k$, two different choices of generators are allowed for $\abs{\mfa_k}$ used in $\vp$ and in $\tilde{\vp}$ respectively, following the previously made convention for $\abs{\mfa_k}$.

\begin{remark1}\label{del}

   Theorem~\ref{asymptotic} and Corollary~\ref{vequiv} completely generalize the previously known `global' special case when the Siu psh functions $\vp$ and $\tilde{\vp}$ are given as local weight functions of two singular hermitian metrics of Siu type of a big line bundle on a projective complex manifold: see the comment before \cite[(1.1)]{K15} where    \cite[Theorem 1.11]{DEL} together with Theorem~\ref{openness} was used for this.

\end{remark1}

\begin{remark1}

 For our purpose in this paper, it is sufficient to define a Siu psh function only on $U \subset X$ as in \eqref{siu} where each ideal $\mfa_k$ has a chosen finite set of generators. More generally, when each of $\mfa_k$ is an ideal sheaf and $X$ is covered by such $U$'s, there are two options. One is to use the global version of a Siu psh function, namely a singular hermitian metric of a line bundle which is locally of the form \eqref{siu}. The other is only to define Siu psh functions on each $U$ as above, and in the intersections, different Siu psh functions will be v-equivalent as in Corollary~\ref{vequiv}. In either case, Theorem~\ref{asymptotic} still makes sense.

\end{remark1}

 \begin{remark1}

   Theorem~\ref{asymptotic} has the pleasant consequence that all the types of algebraic multiplier ideals appearing in the standard algebraic theory of multiplier ideals as in \cite[Part Three]{L} can now be considered in the uniform setting of psh functions (and singular hermitian metrics of line bundles which have psh local weight functions) : namely, multiplier ideals associated to an effective $\QQ$-divisor, to an ideal sheaf, to a linear system and asymptotic multiplier ideals. See Definitions (9.2.1), (9.2.3), (9.2.10), (11.1.2), (11.1.15), (11.1.24) in \cite{L}.

 \end{remark1}

  \subsection{Toric psh functions and Newton convex bodies}\label{toricpsh}

 A psh function $\vp(z_1, \ldots, z_n)$ is called \emph{toric} (or \emph{multi-circled} in \cite{R13b}) if the value depends only on $\abs{z_1}, \ldots, \abs{z_n}$, i.e.

\noi  $\vp(z_1, \ldots, z_n) = \vp (\abs{z_1}, \ldots, \abs{z_n})$. (See e.g. \cite{G} and \cite{R13b} for more information on toric psh functions.) Toric psh functions play an important role in this paper. For the purpose of this paper, we may assume that the domain of toric psh functions in the following is the unit polydisk $\DD^n = D(0,1) \subset \CC^n$ (i.e. with the center at the origin and the polyradius $(1,\ldots,1)$).

 It is well known that given a toric psh function $\vp(z_1, \ldots, z_n)$ on $\DD^n$, one has a convex function $\widehat{\vp}$ on $\RR^n_{-}$, non-decreasing in each variable, associated to $\vp$ such that $$\vp(z_1, \ldots, z_n) = \widehat{\vp} ( \log\abs{z_1}, \ldots, \log\abs{z_n}) .$$

 Moreover, there is a naturally associated closed convex subset $P(\vp) \subset \RRn$ satisfying $P(\vp) + \RRn \subset P(\vp)$ that generalizes the Newton polyhedron associated to a monomial ideal. See \cite[(3.1)]{R13b} for its construction where it is called the indicator diagram of $\vp$. We call $P(\vp)$ the Newton convex body of $\vp$. (This $P(\vp)$ is the closure of what is called the Newton convex body of $\vp$ in \cite{G}, see \cite[Lemma 1.19]{G}.)  The key property of $P(\vp)$ we will use is the following characterization of multiplier ideals.

\begin{proposition}\label{multi}\cite[Theorem A]{G}, \cite[Proposition 3.1]{R13b}

Let $\vp = \vp(z_1, \ldots, z_n)$ be a toric psh function. Then the multiplier ideal $\JJ(\vp)$  is a monomial ideal such that $z_1^{a_1} \ldots z_n^{a_n} \in \JJ( \vp)$ if and only if $(a_1 + 1, \ldots, a_n +1) \in \interior P(\vp) $.

\end{proposition}

 Now given a closed convex subset $P \subset \RRn$ satisfying $P + \RRn \subset P$, one can associate a graded system of monomial ideals as in \cite{M02}, \cite{ELSV}, \cite[\S 8]{JM}, as follows:

  \begin{example}\label{graded}

 Make a choice of an infinite sequence of closed rational convex polyhedra satisfying

  (1) $ R_1 \subseteq R_2 \subseteq R_3 \subseteq \cdots \subseteq P$,

  (2) $R_m + \RRn \subset R_m$ for every $m$ and

  (3) $P = \bigcup R_m$.

  Then $E_k := k \cdot R_k$ satisfies $E_l + E_m \subseteq E_{l + m} \quad \text{for all $l,m \ge 1$}$.

  Let $\mathfrak{a}_k$ be the ideal generated by all monomials in $\CC [x_1, \ldots, x_n]$ whose exponents are contained in $E_k$. Then the family $\mathfrak{a}_\bullet = (\mathfrak{a}_k)$ is a graded system of monomial ideals. Now take a Siu  psh function associated to $\mfab$:

 \begin{equation}\label{siu2}
  \vp = \vp_{\mfab} = \log \left( \sum^{\infty}_{k \ge 1}  \ep_k  \abs{\mfa_k}^{\frac{1}{k}} \right)
 \end{equation}

\noindent where $\ep_k$ is a choice of a sequence of nonnegative coefficients to make the series converge.

  \end{example}

  Using the above, we can show that given a Newton convex body, there exist many toric psh functions sharing the same Newton convex body using the following proposition. Note that  $  \vp = \vp_{\mfab}$  in the above example is a toric psh function.

\begin{figure}[t]
\label{figure1}
\centering
\begin{subfigure}{0.45\textwidth}
    \includegraphics[width=0.9\textwidth]{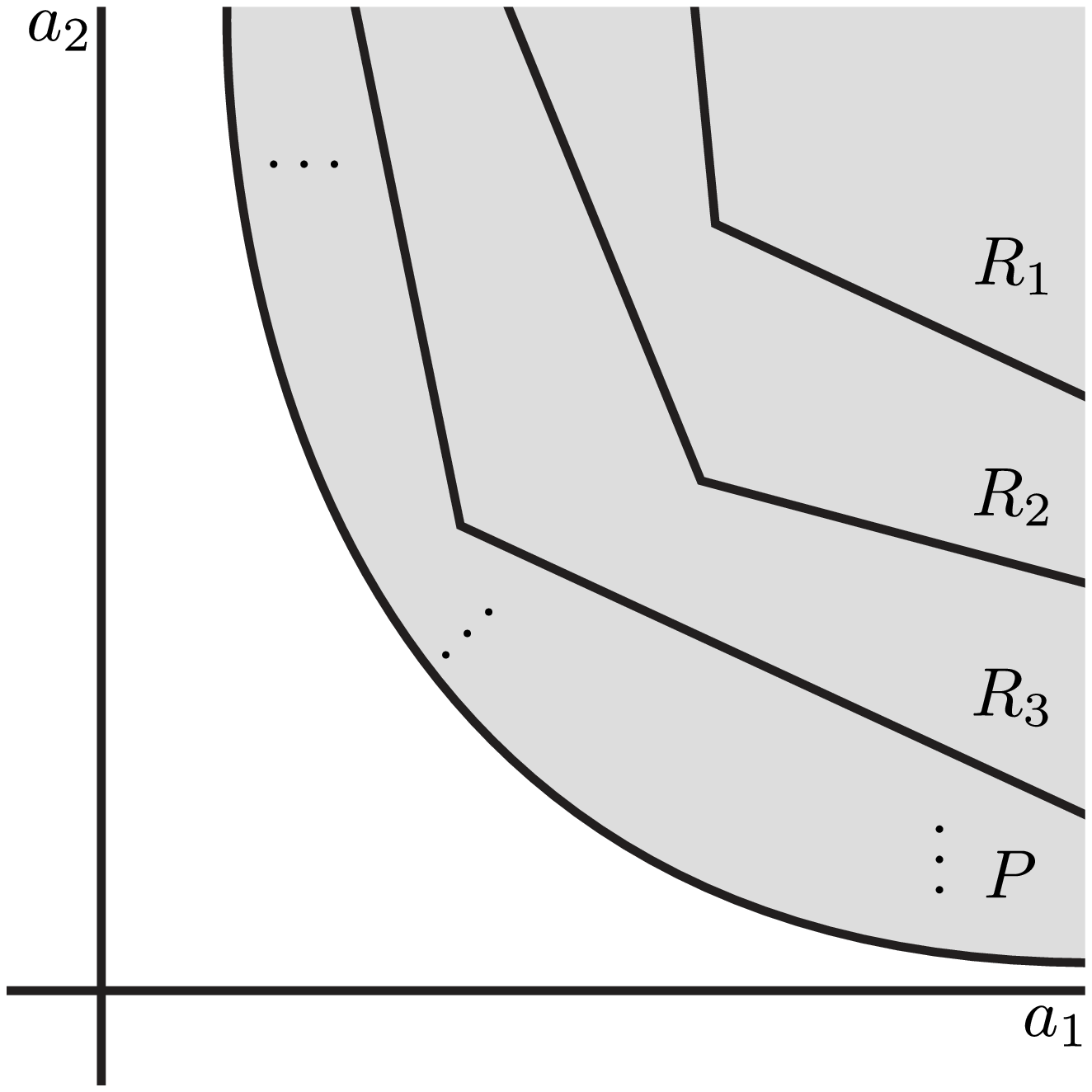}
\end{subfigure} \hfill
\begin{subfigure}{0.45\textwidth}
    \includegraphics[width=0.9\textwidth]{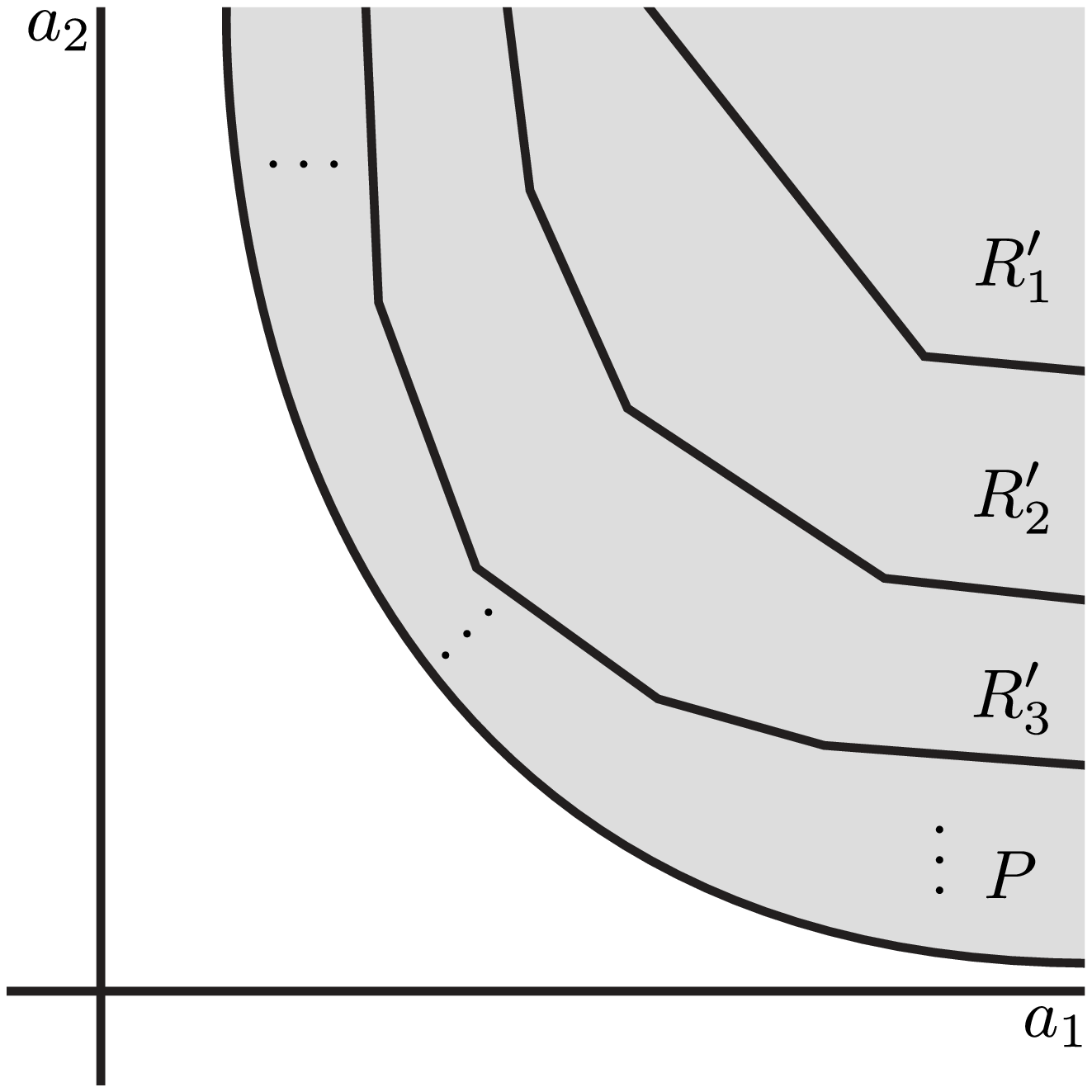}
\end{subfigure}
\caption{Two different choices of a sequence of closed rational convex polyhedra for $P$ in Example~\ref{graded}}
\end{figure}

 \begin{proposition} \label{toric_convex}

  Let $P \subset \RRn$ and $\vp_{\mfab}$   be    as   in Example~\ref{graded}.  The Newton convex body of $\vp_{\mfab}$ is equal to $P$.

 \end{proposition}

 \begin{proof}

  Let $Q$ be the Newton convex body of $\vp_{\mfab}$. If $Q$ is not equal to $P$, then for some sufficiently large $m > 0$, the multiplier ideals  $\JJ(m \vp_{\mfab})$ and $\JJ(m \mfab) $  must be different monomial ideals since we have Howald type theorems (cf.\cite[9.3.27]{L}) for $\JJ(m \vp_{\mfab})$ by (\ref{multi}) and for the asymptotic multiplier ideals $\JJ(m \mfab) $ by \cite[Proposition 8.4]{JM}, respectively. Such difference contradicts Theorem~\ref{asymptotic}.

 \end{proof}

 As a consequence of Proposition~\ref{toric_convex}, we obtain at once so many examples of toric psh functions sharing the same Newton convex body, yet no pair of which is equivalent to each other for at least two reasons. One is from \cite[Theorem 3.5]{K15}, i.e. from different choices of the coefficients $\ep_k$'s.

 Another is from different choices of a sequence of closed rational convex polyhedra $R_1 \subset R_2 \subset \ldots$ and $R'_1 \subset R'_2 \subset \ldots$ as in Figure 1. For reference, we explicitly write out two of those psh functions below again as in \eqref{siu}. The ideal $\mfa'_k$ is the one generated by all monomials whose exponents are contained in $E'_k := k \cdot R'_k$ as in Example~\ref{graded}.

  \begin{equation*}
   \vp_{\mfab} = \log \left( \sum^{\infty}_{k \ge 1}  \ep_k  \abs{\mfa_k}^{\frac{1}{k}} \right)  = \log \left( \sum^{\infty}_{k \ge 1}  \ep_k   \left( \abs{g_1^{(k)}} + \ldots + \abs{g_m^{(k)}}  \right)^{\frac{1}{k}} \right)
 \end{equation*}

  \begin{equation*}\label{siu3}
   \vp'_{\mfab} = \log \left( \sum^{\infty}_{k \ge 1}  \ep'_k  \abs{\mfa'_k}^{\frac{1}{k}} \right)  = \log \left( \sum^{\infty}_{k \ge 1}  \ep'_k   \left( \abs{{g'}_1^{(k)}} + \ldots + \abs{{g'}_{m'}^{(k)}}  \right)^{\frac{1}{k}} \right)
 \end{equation*}

   Since  the multiplier ideals of all these psh functions are determined by the same Newton convex body, they all have the same multiplier ideals $\JJ (m \vp)$ for all $m > 0$, i.e. they are v-equivalent : $\vp_{\mfab}  \sim_v \vp'_{\mfab} $.

\section{Jumping numbers}

\subsection{Definition and basic properties}

 We will define jumping numbers and prove some basic properties. The following definition is possible thanks to the openness theorem of \cite{GZ}.

\begin{definition}
	
 Let $\vp$ be a psh function on a complex manifold $X$. A real number $\al > 0$ is a {\bf jumping number} of $\vp$ at $x \in X$ if the multiplier ideals $\JJ(c \vp )$ are constant at $x$ precisely for $c \in [\al, \al + \delta)$ (for some $\delta > 0$) in the sense that $\JJ(c \vp)_x = \JJ(\al \vp)_x$ for $ c \in [\al , \al + \delta)$ and $ \JJ((\al + \delta) \vp)_x \subsetneq \JJ(\al \vp)_x .$ 		Denote the set of those jumping numbers by $  \jump(\vp )_x$.	

\end{definition}

 One may also consider a global version of the set of jumping numbers for a compact complex manifold $X$ defining  $ \jump(\vp)_X := \displaystyle \bigcup_{x \in X} \jump(\vp)_x $  (cf. \cite{D15}). Note that $\jump(\vp)_x$ and $\jump(\vp)_X$ all satisfy DCC due to the openness theorem (\ref{openness}). Since any set of real numbers satisfying DCC is countable at most, the sets $\jump(\vp)_x$ and $\jump(\vp)_X$ are countable.

 In this paper, we will concentrate on the study of the jumping numbers at a point as in \cite{ELSV}, which then can be applied to study the global version $\jump(\vp)_X$.  As in \cite{D15}, $\jump(\vp)_X$ is necessary when one wants to keep track of all the (possibly non-reduced) subschemes associated to $\JJ(c \vp)$ for all $c > 0$.

 Now for basic properties of jumping numbers, we will first generalize \cite[(1.17) and (5.8)]{ELSV} to general psh functions. Let $\vp$ be a psh function on a complex manifold $X$. Let $x \in X$ be a point where the Lelong number of $\vp$, $\nu (\vp, x) > 0$. Let $c(\vp, x)$ be the log canonical threshold of $\vp$ at $x$.
 	
 	Let $\chi$ be one of the jumping numbers of $\vp$ at $x$ : i.e. $\chi \in \jump(\vp)_x$. By the openness theorem  (\ref{openness}), there exists the `next' jumping number after $\chi$, which we denote by $\chi'$ in the following statement.  	
 	
 \begin{proposition}
 	
 $\chi' \le \chi + c(\vp, x)$.

 \end{proposition}

\begin{proof}
	
	It is enough to show that $ \JJ \left( (\chi + c(\vp, x)) \vp \right)_x \neq \JJ(\chi \vp)_x.$  By the subadditivity of multiplier ideals \cite[Theorem 2.6]{DEL}, we have $ \JJ( (\chi + c(\vp, x)) \vp ) \subseteq \JJ( \chi \varphi) \cdot \JJ( c(\vp, x) \vp ). $ Since $\JJ ( c(\vp, x) \vp ) \neq \OO_{X,x}$, the conclusion follows.
	
\end{proof}

Using (\ref{multi}), we next find a relation between the jumping numbers of a toric psh function $\vp$ at the origin and its Newton convex body. Denote by $C(\vp)$, the set of positive real numbers $c$ such that $\partial P(c\vp) \cap \ZZ_{>0}^n$ is not empty.

\begin{proposition}\label{jump_newton}

    Let $\vp$ be a toric psh function on the unit polydisk $\DD^n \subseteq \CC^n$. Then $C(\vp) \subseteq \jump(\vp)_0$ and $\jump(\vp)_0$ is the closure of $C(\vp)$  in $\RR$.

\end{proposition}

\begin{proof}
    For $c \in C(\vp)$, we can find $A+\mathbf{1} = (a_1 + 1, \ldots, a_n +1) \in \partial P(c\vp) \cap \ZZ_{>0}^n$. Since $A+\mathbf{1} \in \interior P((c- \epsilon)\vp)$ for $0 < \epsilon \ll 1$, $z^A \in \JJ( (c-\epsilon)\vp )$ but $z^A \notin \JJ( c\vp)$. Therefore, $c$ is a jumping number of $\vp$ at $0$.

    For the second statement, it is enough to show that the set of jumping numbers which are not cluster points is contained in $C(\vp)$. If $c>0$ is a jumping number of $\vp$ at $0$ and it is not a cluster point of $\jump(\vp)_0$, then one can find $z^A = z_1^{a_1} z_2^{a_2} \cdots z_n^{a_n} \in \JJ((c-\epsilon)\vp) \backslash \JJ(c \vp)$ for sufficiently small $\epsilon > 0$ (where $A$ is independent of $\epsilon$). We have
	$$A + \mathbf{1} \in \interior P((c-\epsilon)\vp) \quad \text{but} \quad A + \mathbf{1} \notin \interior P(c\vp).$$
	Since the intersection of $\interior P((c-\epsilon)\vp)$ for $0 < \epsilon \ll 1$ is $P(c\vp)$, we have $\partial P(c \vp) \cap \ZZ_{>0}^n \neq \emptyset$.
\end{proof}

 On the other hand, the following generalizes \cite[Lemma 5.12]{ELSV}.

\begin{proposition}\label{isolated}

(1) Let $\vp$ be a psh function with isolated singularities at $p \in X$. Then the set of its jumping numbers at a point, does not have a cluster point.

(2) Let $\vp \in \EE (\Omega)$ be a psh function in the Cegrell class.  Then the set of its jumping numbers at a point $x \in \Omega$, does not have a cluster point.

\end{proposition}

 A psh function in the Cegrell class is the ultimate generalization of a psh function with isolated singularities : see \cite{C04}.

\begin{proof}

 Let $\mathfrak{m}$ be the maximal ideal of the point $p$. For $\el > 0$, the multiplier ideal $\JJ ( \el \vp)$ is $\mathfrak{m}$-primary. Thus the colength of $\JJ (\el \vp)$, $\dim \OO_X \slash \JJ(\el \vp) $ gives an upper bound for the possibilities of different ideals $\JJ (a \vp)$ with $0< a < \el$ since $ \dim \OO_X \slash \JJ(\el \vp) \ge \dim \OO_X \slash \JJ(a \vp) .$

 For (2), the same argument holds since in this case, the stalks of the multiplier ideals $\JJ(\el \vp)$ are either trivial or $\mathfrak{m} = \mathfrak{m}_x$-primary at each point $x \in \Omega$ (see \cite[(3.2)]{KR}).

\end{proof}

\noi We note that this proposition can be used to decide whether a psh function belongs to the Cegrell class or not.

 Now we give some examples of jumping numbers.

\begin{example}\cite{ELSV}\label{5.10}

 	Consider the convex region $P = \{ (x,y) \in \mathbf{R}_{> 0}^2 : (x-1)(y-1) \ge 1 \}.$ Apply the construction of Example~\ref{graded} to obtain a Siu psh function $\varphi_{\mathfrak{a}_\bullet}$ on $\DD^2 \ni (0,0)$. 	By Theorem~\ref{asymptotic}, the jumping numbers of $\varphi_{\mathfrak{a}_\bullet}$ are precisely the jumping numbers of $\mathfrak{a}_\bullet$  given in \cite[Example~5.10]{ELSV} :
\begin{equation}\label{5.10phi}	
	\Jump(\vp_{\mathfrak{a}_\bullet})_{(0,0)} = \left\{ \frac{ef}{e+f} : e,f \ge 1 \right\} .
\end{equation}		
	
\end{example}

 In this example, all positive integers are cluster points of jumping numbers. On the other hand, the following example  illustrates that in general, the jumping numbers of $\vp + \psi$ may not be simply described in terms of the jumping numbers of $\vp$ and $\psi$.

\begin{example}[Jumping numbers of $\JJ( c (\vp + \psi) )$] \label{vp+psi}
    Let $\vp = \vp_{\mfab}$ be as in Example~\ref{5.10} and let $\psi = \log \abs{z_1}$.  The Newton convex body of $c (\vp + \psi)$ is $ P( c(\vp + \psi) ) = c ( P(\vp) + P(\psi) ) $ by Lemma~\ref{Minkowski}. The boundary of $ P( c(\vp + \psi) )$  is given by the equations $(x - 2c)(y - c) = c^2$ and $x > 2c$. Using Proposition~\ref{jump_newton}, we find that the set of jumping numbers for $\vp + \psi$ is given by

    $$ \jump( \vp + \psi; 0)_{(0,0)} = \mathrm{cl} \left(\left\{ \frac{ (p+2q) - \sqrt{p^2 + 4q^2} }{2} : p, q \in \ZZ_{>0} \right\} \right), $$
    where $\mathrm{cl}(A)$ is the closure of a  subset $A \subset \RR$ in $\RR$.
\end{example}

 The following lemma was used in the above example. (A. Rashkovskii informed us that Lemma~\ref{Minkowski} might be already known. In this regard, see Section 8 of \cite{R09}.) 

\begin{lemma}\label{Minkowski}

    Let $\vp$ and $\psi$ be toric psh functions on the polydisk $D(0,r) \subseteq \CC^n$.
    Then the Newton convex body of $\vp+\psi$ is the Minkowski addition of $P(\vp)$ and $P(\psi)$, i.e. $P (\vp + \psi) = P(\vp) + P(\psi)$.
\end{lemma}

\begin{proof}
    Let $f_1$ and $f_2$ be convex functions on $(-\infty, \log{r})^n$, non-decreasing in each variable, associated to $\vp$ and $\psi$ respectively. We may assume that $f_1$ and $f_2$ are defined on $\RR^n$ by putting $f_1(x) = f_2(x) = \infty$ whenever $x \notin (-\infty, \log{r})^n$. Also replacing $f_1$ and $f_2$ by their lower semi-continuous regularizations, we may assume that $f_1$ and $f_2$ are lower semi-continuous on $\RR^n$. Let $\tilde{f}_1$, $\tilde{f}_2$ and $\widetilde{f_1+f_2}$ be the Legendre transforms of $f_1$, $f_2$ and $f_1+f_2$. Then by \cite[Theorem 2.2.5]{H}, $\widetilde{f_1+f_2}$ is the lower semi-continuous regularization of
    $$ g(y) = \inf_{y_1 + y_2 = y}{ \left( \tilde{f}_1(y_1) + \tilde{f}_2(y_2) \right) }, $$
    which implies that $y \in P(\vp + \psi)$ if and only if there exist $y_1 \in P(\vp)$ and $y_2 \in P(\psi)$ such that $y_1 + y_2 = y$.
\end{proof}

\begin{remark1}
In the proof of Lemma~\ref{Minkowski}, we replaced the Newton convex bodies by their interiors so that the lower semi-continuous regularizations of $f_1$, $f_2$ and $g$ are identically equal to the original functions on their Newton convex bodies. This does not affect the conclusion since the lower semi-continuous regularization changes values only on the boundary of the Newton convex bodies.
\end{remark1}

\subsection{``Mixed" jumping numbers }

 It is also very natural to consider the following more general type of jumping numbers arising from the multiplier ideals $\JJ(c \vp + \psi)$ for psh functions $\vp$ and $\psi$. In particular, this type of jumping numbers is used in a recent work of Demailly\cite{D15} on $L^2$ extension theorems from non-reduced subschemes defined by multiplier ideal sheaves. These jumping numbers play an important role there to keep track of all the non-reduced subschemes defined by  $\JJ(c \vp + \psi)$ for all $c > 0$.

\begin{definition}
	
 A real number $\al > 0$ is a {\bf jumping number of $\vp$ at $x$ with respect to $\psi$} if the multiplier ideals $\JJ(c \vp + \psi)$ are constant at $x$ precisely for $c \in [\al, \al + \delta)$ for some $\delta > 0$. 	Denote the set of them by $   \Jump(\vp ; \psi)_x$.	
 \end{definition}

\noi  Note that $ \Jump(\vp ; \psi)_x = \Jump(\vp)_x $ when $\psi$ is locally bounded, according to our previous notation.

 In general, the jumping numbers of $\JJ(c\vp)$ and those of $\JJ(c\vp + \psi)$ look quite different from each other as seen already in the algebraic case :  Let $D$ and $F$ are effective divisors on $X$. Let $f: X' \to X$ be a log-resolution of $(X, D+ F)$. Let $f^* D = \sum r_i E_i$, $f^* F = \sum s_i E_i$ and $K_{X'/X} := K_{X'} - f^*(K_X) = \sum b_i E_i$ (i.e. define the coefficients $r_i, s_i, b_i$ by these equalities). 

 Then the jumping numbers of $(X,cD)$ are contained in $$ A:= \left\{ \frac{b_j + m}{r_j} \right\}$$ (where $m \ge 1$ are integers).  On the other hand, the jumping numbers of $(X,cD + F)$ are contained in

 $$ B:= \left\{ \frac{b_j + m - s_j}{r_j} \right\} .$$

In other words, these sets of numbers are candidate jumping numbers and not all of them are actual jumping numbers (see e.g. \cite{ST07}). In any case, those two sets $A$ and $B$ are rather different set of rational numbers except when $\frac{s_j}{r_j}$ is constant for every $j$, in which case the one set of jumping numbers are translates of the other.

 For general psh functions with not necessarily analytic singularities \cite[(1.10)]{D11}, $s_j$ and $r_j$ are generalized as the values of divisorial valuations of the psh functions.  Thus we are lead to the following proposition.

 \begin{proposition}
	
 Suppose that psh functions $r \vp$ and $\psi$ are v-equivalent for some constant $r >0$. Then the elements of  $\Jump(\vp; \psi)_x$ are `translates' by $-r$ of the elements in $\Jump(\vp ; 0)_x$.
	
\end{proposition}

\begin{proof}
	
	We have $ (m + r) \vp \sim_v m \vp + \psi$.  Then for the multiplier ideals, we get (by \cite{BFJ}, \cite{GZ}) $\JJ( (m+r) \vp) = \JJ( m\vp + \psi)$ for every real $m  > 0$.

\end{proof}

For mixed jumping numbers, we first record the following example where $\jump(\vp;0)_x$ and $\jump(\vp;\psi)_x$ are equal while (multiples of) $\vp$ and $\psi$ are not v-equivalent.

\begin{example}[Jumping numbers of $\JJ(\psi + c \vp)$] \label{rel_jumping}

Let $\vp = \vp_{\mfab}$ on $\CC^2$ be as in Example~\ref{5.10}. Let $\psi = \log{\abs{z_1}}$. Since $\psi + c\vp$ is toric psh, we can use Proposition~\ref{multi} to determine $\JJ(\psi + c\vp)$. The Newton convex body of $\psi + c\vp$ is given by

$$ P(\psi + c\vp) = \{ (x,y) \in \RR_{\ge 0}^2 : (x-1-c)(y-c) \ge c^2, x \ge 1+c \}.  $$

\noi By Proposition~\ref{jump_newton}, a positive real number $c$ is an element of $\jump(\vp;\psi)_{(0,0)}$ if and only if there exist two positive integers $p$ and $q$ such that  $(p,q)$ is on the boundary of $P(\psi + c \vp)$. Therefore we have
$$ \jump(\vp;\psi)_{(0,0)} = \left\{ \frac{ ef }{ e+f } : e, f \ge 1 \right\} = \jump(\vp; 0)_{(0,0)} . $$

\end{example}

 In the following example, $\jump(\vp;\psi)_x$ are translates by $-1$ of $\jump(\vp; 0)_x$ while (multiples of) $\vp$ and $\psi$ are not v-equivalent. Note that v-equivalence of toric psh functions is determined by the Newon convex bodies due to (\ref{multi}).

\begin{example}[Jumping numbers of $\JJ(\psi + c \vp)$] \label{rel_jumping2}
Consider two convex sets $P_1$ and $P_2$ given by
$$ P_1 = \{ (x,y) \in \RR_{\ge 0}^2 : xy \ge 1 \}, \quad P_2 = \{ (x,y) \in \RR_{> 0}^2  : (x-1)(y-1) \ge 1 \}. $$

\noi From Proposition~\ref{toric_convex}, there exist psh functions $\vp$ and $\psi$ such that $P(\vp) = P_1$ and $P(\psi) = P_2$. Then we have
$$\Jump(\vp)_{(0,0)} = \left\{ \sqrt{k} : k \in \ZZ_{>0} \right\} \quad \text{and} \quad \Jump(\psi)_{(0,0)} = \left\{ \frac{ef}{e+f} : e,f \in \ZZ_{>0} \right\} .$$

\noi By Lemma~\ref{Minkowski}, we have
$$ P(\psi + c \vp) = \{ (x,y) \in \RR_{\ge 0}^2 : (x-1)(y-1) \ge (c + 1 )^2 \} $$
for $c > 0$. Therefore we obtain
$$ \jump(\vp;\psi)_{(0,0)} = \left\{ \sqrt{k} - 1 : k \in \ZZ_{>0}, k \ge 2 \right\}.$$

\end{example}

 \section{Periodicity of jumping numbers fails}

\qa In \cite[Proposition A, Proposition 1.12]{ELSV}, they showed that in the algebraic case, jumping numbers have a periodic behavior in the algebraic case. In \cite{ELSV}, a specific definition of periodicity was not given. The following is a general definition of periodicity which generalize the phenomenon of \cite{ELSV} Proposition 1.12 and Proposition A for general psh functions.

 \begin{definition}

 The set of jumping numbers $J := \Jump(\vp; \psi)_x$ (or more generally, a set of nonnegative real numbers) is said to have a \textbf{period} $c>0 \; (c \in \RR)$ if the following holds:  for every $\al \in J$, there exists an integer $m_0 = m_0 (\al, J) \ge 1$ such that $m \ge m_0$ implies $\al + m c \in J$.
 	
 \end{definition}

Note that there can be possibly more than one period for $\Jump(\vp; \psi)_x$ as easily seen from the following example which is a psh version of \cite[Example 1.9, Example 5.1 (iii)]{ELSV} : \textit{diagonal ideals}.

\begin{example}\label{diagonal}

	For positive real numbers $m_1, \ldots, m_n$, consider the psh function $\varphi = \log{(|z_1|^{m_1} + \cdots + |z_n|^{m_n})}. $ The jumping numbers of $\varphi$ at the origin $0 \in \CC^n$ consist of
	
	$$\Jump (\vp)_ 0 = \left\{ \frac{e_1 + 1}{m_1} + \cdots + \frac{e_n + 1}{m_n} : e_1, \ldots , e_n \in \ZZ_{\ge 0} \right\}.$$

\end{example}
\qa
\\

  Now we will show by an example that the periodicity of jumping numbers does not hold for general psh functions. Takayuki Koike~\cite{Ko15} gave an example of a psh function $\vp$ with the following set of jumping numbers at a point. The psh function $\vp$ appears as a local weight function of a singular hermitian metric for a holomorphic line bundle  on some smooth projective variety $X$ (given by \cite{N04}) of dimension $4$ (see \cite[Theorem 1.1]{Ko15}).

\begin{example}\label{koike}

 Let $x \in X$ and $\vp$ be as in \cite[Theorem 1.1]{Ko15} so that

 $$ \Jump(\vp)_x = \left\{ \frac{1}{2} ( p + \sqrt{2a^2 p^2 - q^2} ): p, q \in \ZZ, p > q \ge 0, \; p \equiv q \mod 2 \right\} $$

\noi where $a \ge 2$ is a fixed integer.

%\todo{Changed the size of braces}

\end{example}

 It was mentioned in \cite[p.299]{Ko15} that ``it seems difficult to expect the periodicity property" for this example. We confirm that it is indeed the case.

\begin{proposition}

  For the psh function $\vp$ at $x$ in Example~\ref{koike}, the set of jumping numbers $ \Jump(\vp)_x$  does not have a period.

\end{proposition}

\begin{proof}

 It suffices to show that $J:= 2 \Jump (\vp)_x$ has no periods. Assume that $J$ has a period $c$. Then there exist positive integers $\alpha$ and $m$ such that $\alpha$, $\alpha + mc$ and $\alpha + 2mc$ are in $J$ since we can take $$\alpha      = p_1 + \sqrt{2a^2 p_1^2 - q_1^2} $$ where   $p_1 = z$ and $q_1 = a\abs{x-y}$ for some $x,y,z \in \ZZ_{>0}$ satisfying  $x^2 + y^2 = z^2$. Since such a Pythagorean triple $(x,y,z)$ can be written as
    $$ x = 2nk + k^2, \quad y = 2n^2 + 2nk, \quad z = 2n^2 + 2nk + k^2 $$
    for positive integers $k, n$, we have
    $$ \frac{z}{|x-y|} = \frac{ \left( \frac{k}{n} \right)^2 + \frac{2k}{n} + 2 }{ \left| \left( \frac{k}{n} \right)^2 - 2 \right| }. $$

\noi    Therefore we get $p_1 > q_1$ when  we take $\frac{k}{n}$ sufficiently close to $\sqrt{2}$. Also  $p_1 \equiv q_1 (\textrm{mod }2)$ holds when we take $k$ even. Then by the definition of $J$, we can write
    \begin{align*}
        \alpha + mc &= p_2 + \sqrt{2a^2 p_2^2 - q_2^2}, \\
        \alpha + 2mc &= p_3 + \sqrt{2a^2 p_3^2 - q_3^2}
    \end{align*}
    for some $p_i, q_i$. Thus we have
    \begin{equation} \label{Koike1}
        2p_2 -p_3 - \alpha
        = \sqrt{ 2a^2 p_3^2 - q_3^2 } - 2\sqrt{ 2a^2 p_2^2 - q_2^2 }.
    \end{equation}
    Since the left hand side of (\ref{Koike1}) is an integer, the right hand side is also an integer.  Observe that if $\sqrt{a} - \sqrt{b}$ is an integer for two integers $a$ and $b$, then both $\sqrt{a}$ and $\sqrt{b}$ are integers. With this observation,
    $
        \sqrt{ 2a^2 p_3^2 - q_3^2 }\ \ \text{and} \ \
        2\sqrt{ 2a^2 p_2^2 - q_2^2 }
    $
    are also integers, which implies that $c$ is rational.

    Now take a positive integer $n$ such that $nc$ is an integer.
    Let  $\beta \in J$ be irrational. There exists a positive
    integer $m_0$ such that $\beta + mc \in J$ for all $m \ge m_0$.
    For $m \ge m_0$, we have $\beta + (mn)c \in J$.
    Hence for some $p,\ q,\ r,\ s$, we have
    \begin{align*}
        \beta &= p + \sqrt{ 2a^2 p^2 - q^2 }, \\
        \beta + (mn)c &= r + \sqrt{ 2a^2 r^2 - s^2 }.
    \end{align*}
    This implies that $\sqrt{ 2a^2 p ^2 - q^2 } - \sqrt{ 2a^2 r^2 - s^2 }$ is
    an integer and thus
    $\sqrt{ 2a^2 p ^2 - q^2 }$ and  $\sqrt{ 2a^2 r^2 - s^2} $
    are integers as well.
    This contradicts to irrationality of $\beta$. Hence  the period $c$ cannot exist.
\end{proof}

 Before closing this section, let us consider again Example~\ref{5.10}. If a period $c$ exists in Example~\ref{5.10}, it must be an integer. A possible way to show that a period cannot exist in this case is as follows:  consider
 jumping numbers  $\frac{e}{e+1}$ converging to $1$. If a period $c$ exists, $\frac{e}{e+1} + c$ must be all in the set of jumping numbers. They are all having difference $\frac{e}{e+1}$ from the integer $c$. It reduces to the following arithmetic question whose negative answer will imply that a period cannot exist in this case.

\begin{question}

 Is it true that for every integer $e , c \ge 1$, one can find integers $r, s \ge 1$ such that the following holds $$ \frac{e}{e+1} + c = \frac{rs}{r+s} .$$

\end{question}

\section{Cluster points of jumping numbers }

 In this section, we will study the `bad' behaviour of jumping numbers in the general psh case given by cluster points. Recently Guan and Li \cite{GL} gave the following example of a psh function with non-analytic singularities such that it has a cluster point of jumping numbers:

\begin{equation}\label{glpsh}
 \vp(z_1, z_2) = \log \abs{z_1} + \sum^{\infty}_{k=1} a_k \log ( \abs{z_1} + \abs{\frac{z_2}{k}}^{b_k} )  =: \log \abs{z_1}  + \psi
 \end{equation}

 \noi which is toric psh, where  $a_k = M^{-k}, b_k = M^{2k} \quad (M \ge 2) $.

 \begin{theorem}[Guan and Li]\cite[Theorem 1.1]{GL}\label{guanli}

 In the set of jumping numbers $\jump(\vp)_0$ at the origin,  there exists a sequence $(c_k)_{k \ge 1}$ converging to $1$, i.e. $1$ is a cluster point of jumping numbers.

 \end{theorem}

  On the other hand, a cluster point of jumping numbers was known in the context of graded system of ideals from \cite[Example 5.10]{ELSV}. Thanks to Theorem~\ref{asymptotic}, this gives another example of a psh function with a cluster point of jumping numbers as noted in Example~\ref{5.10}.

  In this section, we will completely generalize these two particular examples in Theorem~\ref{cluster_asymptote}.

   \subsection{Cluster points of jumping numbers for toric psh functions}

 We first note the following general property.

\begin{proposition}\label{cluster}
	
Let $c$ be a cluster point of the set of jumping numbers $\Jump(\vp ; \psi)_x$. Then $c$ itself is also a jumping number in $\Jump(\vp; \psi)_x$.
	
\end{proposition}

\begin{proof}
	
 Suppose not. Then there exists $\ep > 0$ such that for $d \in (c-\ep, c+\ep)$, the multiplier ideals $\JJ(d \vp + \psi)$ are constant. This contradicts to $c$ being a cluster point.

 \end{proof}

  When a psh function is toric, we answer Question~\ref{clusters} by the following

 \begin{theorem}\label{mtimes}
 	
 	Let $\vp$ be a toric psh function on the unit polydisk $\DD^n \subset \CC^n$. If $c$ is a  jumping number at the origin $0 \in \DD^n$, then for each integer $m \ge 1$, $m c$ is also a jumping number in $\Jump (\vp)_0$.

 \end{theorem}

\noi  We first note that for the individual examples of toric psh functions in Example~\ref{5.10}, Example~\ref{diagonal}, Example~\ref{koike}, the validity of  this theorem can be easily checked.

\begin{proof}
	
	Let $P(c \vp)$ be the Newton convex body of $c \vp$. Assume that $c$ is a jumping number of $\vp$ at $0$. By Proposition~\ref{jump_newton}, it is enough to show that $mC(\vp) \subset C(\vp)$. If $c \in C(\vp)$, we have, since $\partial P( c\vp) \cap \ZZ_{>0}^n \neq \emptyset$,
	$$ \partial P( (mc) \vp) \cap \ZZ_{>0}^n \supseteq m \cdot ( \partial P(c \vp ) \cap \ZZ_{>0}^n ) \neq \emptyset, $$
	which implies that $mc \in C(\vp)$.
\end{proof}

Note that the property in Theorem~\ref{mtimes} describes an intrinsic property satisfied by toric psh functions. This can be used to show that a particular psh function is never toric in any coordinates (as in Example~\ref{saito}). The definition of toric psh as in Section~\ref{toricpsh} depends on the choice of coordinates.  An immediate consequence of Theorem~\ref{mtimes} is the following

\begin{corollary}\label{mcmc}
	
 Let $\vp$ be a toric psh function on $\DD^n$. If $c$ is a cluster point of jumping numbers at $0$, then for each integer $m \ge 1$, $mc$ is also a cluster point of $\Jump (\vp)_0$.
	
\end{corollary}

 We remark that the statement of Theorem~\ref{mtimes} does not hold for non-toric psh functions.  Consider the following example of Morihiko Saito  \cite[Example 3.5]{S93} with computation from \cite[Example 6.1]{CL}.  (See also e.g. \cite[\S 2]{ELSV}, \cite{BS05}, \cite{S07}  for more relations between jumping numbers and Bernstein-Sato polynomials.)

 \begin{example}\label{saito}
 	
 	Let $f = x^5 + y^4 + x^3 y^2$ on $\CC^2$. The multiplier ideals $\JJ( c \vp )$ for $\vp = \log{|f|}$ at $(0,0) \in \CC^2$ are given by
 	
 	$$
 	\JJ(c \vp) = {\renewcommand{\arraystretch}{1.5}
 	\left\{
 	\begin{array}{ll} %{*2{>{\displaystyle}l}}
 	    \OO_{\CC^2,(0,0)} & \text{if $0 \le c < \frac{9}{20}$,} \\
 	    (x,y) & \text{if $ \frac{9}{20} \le c < \frac{13}{20}$,} \\
 	    (x^2,y) & \text{if $ \frac{13}{20} \le c < \frac{7}{10}$,} \\
 	    (x^2, xy, y^2) & \text{if $ \frac{7}{10} \le c < \frac{17}{20}$,} \\
 	    (x^3, xy, y^2) & \text{if $ \frac{17}{20} \le c < \frac{9}{10}$,} \\
 	    (x^3, x^2y, y^2) & \text{if $ \frac{9}{10} \le c < \frac{19}{20}$,} \\
 	    (x^3, x^2y, xy^2, y^3) & \text{if $ \frac{19}{20} \le c < 1$,} \\
 	\end{array} \right. }
 	$$
 	where all the ideals are in $\OO_{\CC^2,(0,0)}$ and $\JJ(c \vp) = (f) \cdot \JJ((c-1) \vp)$ for $c \ge 1$.
 %	\todo{ideals of ring}
	 \end{example}

 If the statement of Theorem~\ref{mtimes} holds for this $\vp$, then $\frac{27}{20}$ should also be a jumping number of $\vp$, which is contradiction. Thus the statement of Theorem~\ref{mtimes} does not hold for non-toric psh functions, in general.

 \subsection{Cluster points of jumping numbers for toric psh functions in dimension $2$}

 When we restrict the dimension to $n=2$, we can completely determine the cluster points of jumping numbers for toric psh functions.

    Let $\mathrm{pr}_i : \RR^2 \to \RR$ be the projection to the $i$-th coordinate for $i=1,2$. Set $x_0 := \inf{ \mathrm{pr}_1(P(\vp)) }$ and $y_0 := \inf{ \mathrm{pr}_2(P(\vp)) }$. Define $f,g : \RR \to [0,+\infty] $ as follows:
    \begin{align*}
        f(x) &= \left\{ \begin{array}{ll}
        +\infty & \text{if $x < x_0$,} \\
        \inf{ \{ y \in \RR_{\ge 0} : (x,y) \in P(\vp) \} } & \text{if $x \ge x_0$.}
        \end{array} \right.  \\
        g(y) &= \left\{ \begin{array}{ll}
        +\infty & \text{if $y < y_0$,} \\
        \inf{ \{ x \in \RR_{\ge 0} : (x,y) \in P(\vp) \} } & \text{if $y \ge y_0$.}
        \end{array} \right.
    \end{align*}
    Note that the epigraph $\{ (x,t) \in \RR^2 : x \ge x_0, t \ge f(x) \}$ of $f$ is equal to $P(\vp)$,  thus $f$ is a convex function (see \cite[\textsection  2.1]{H}). Similarly, $g$ is also a convex function. From the inclusion $P(\vp) + \RR_{\ge 0}^2 \subseteq P(\vp)$, we conclude that $f$ and $g$ are decreasing and
    $$ \lim_{x \to + \infty}{ f(x) } = y_0 \quad \text{and} \quad \lim_{y \to +\infty}{ g(y) } = x_0. $$
    From this, we have
\begin{lemma}
    Let $\vp$ be a toric psh function on $\DD^2$. Then, for every $\epsilon >0$, the lines $x = x_0 + \epsilon$ and $y = y_0 + \epsilon$ intersect with the interior of $P(\vp)$ but the lines $x = x_0$ and $y = y_0$ do not intersect with the interior of $P(\vp)$. We call the lines $y=y_0$, $x=x_0$ the horizontal asymptote and the vertical asymptote of $P(\vp)$, respectively.
\end{lemma}

 Recall that $P(\vp)$ is always closed  from its definition. It is easy to see (from $P(\vp) + \RR_{\ge 0}^2 \subseteq P(\vp)$) that the horizontal asymptote $\{ (t, y_0) : t \in \RR\}$ either is disjoint from $P(\vp)$ (as in (f), (h) of Figure 2)  or has  a half-line on it contained in the boundary of  $P(\vp)$ (as in (e), (g) of Figure 2). Similarly for the vertical asymptote as well. The following theorem is Theorem~\ref{dimtwo} in the introduction, which is one of the main results of this paper.

\begin{theorem} \label{cluster_asymptote}
     Assume that $n = 2 $ in the setting of Theorem~\ref{mtimes}. The set of jumping numbers $\Jump (\vp)_0$ has at least one (thus infinitely many, according to Corollary~\ref{mcmc}) cluster point if and only if at least one of the following (1) and (2) holds:
     \begin{enumerate}
         \item $x_0 > 0$ and $\{ (x_0,t) : t \in \RR\} \cap P(\vp) = \emptyset$,
         \item $y_0 > 0$ and $\{ (t, y_0) : t \in \RR\} \cap P(\vp) = \emptyset$.
     \end{enumerate}
     Moreover, the set of cluster points of jumping numbers is precisely equal to
     $$ \left\{ \frac{k}{m} : k \in \ZZ_{>0},\, m \in S \right\} $$
     where $S$ is a subset of $\{x_0, y_0\}$ satisfying the following (a) and (b):
     \begin{enumerate}
         \item[(a)] $x_0 \in S$ if and only if (1) holds,
         \item[(b)] $y_0 \in S$ if and only if (2) holds.
     \end{enumerate}
\end{theorem}

\begin{remark1}
    All the possible asymptotic behaviors of the boundary of $P(\vp)$ are listed in Figure 2. Theorem~\ref{cluster_asymptote} exactly says that the jumping numbers of $\vp$ at $0$ has a cluster point if and only if either (d) or (h) holds.
   In other words, there does not exist a cluster point if and only if, for the Newton convex body $P(\vp)$,  one of (a,b,c) holds \emph{and} one of (e,f,g) holds.
\end{remark1}

\begin{remark1}

 One can also easily see that this theorem explicitly generalizes Proposition~\ref{isolated} in this dimension $2$ toric psh case.

\end{remark1}

\bigskip

\begin{figure}[t]
\label{figure2}
\centering
    \includegraphics[width=0.9\textwidth]{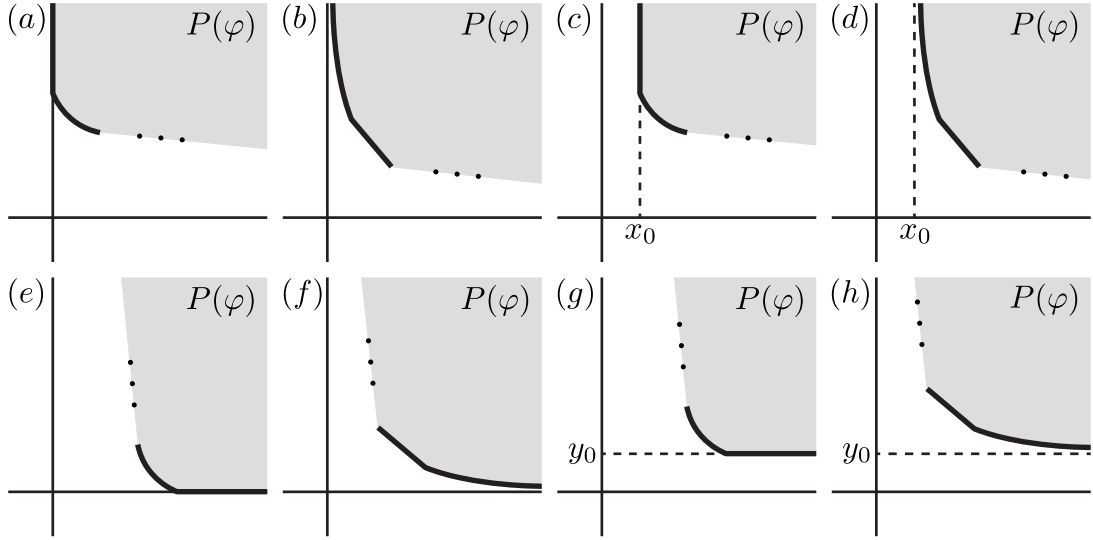}
\caption{All the possible asymptotic behaviors of the boundary of $P(\vp)$ with respect to the two axes.}
\end{figure}

 \begin{proof}[Proof of Theorem~\ref{cluster_asymptote}]
    Assume that $c$ is a cluster point of jumping numbers of $\vp$ at $0$. In view of Proposition~\ref{multi}, we then can find a sequence $(\epsilon_k)$ of positive real numbers and a sequence $(A_k)$ in $\ZZ_{\ge 0}^2$ such that
    \begin{enumerate}
        \item $\epsilon_k > \epsilon_{k+1}$ for every $k \ge 1$ and $\epsilon_k \to 0$,
        \item $A_k + \mathbf{1} \in \interior{P( (c-\epsilon_k) \vp ) }$ but $A_k + \mathbf{1} \notin \interior{P( (c-\epsilon_{k+1}) \vp )}$ for every $k \ge 1$.
    \end{enumerate}
    Note that $A_k + \mathbf{1} \notin \interior{P( c\vp) }$ for every $k \ge 1$ and therefore the first coordinate and the second coordinate of $A_k$ cannot increase simultaneously when $k$ increases. By the second property of $(A_k)$, for each $k \ge 1$, we have either $\mathrm{pr}_1 (A_k) < \mathrm{pr_1}(A_{k+1})$ or $\mathrm{pr}_2(A_k) < \mathrm{pr}_2(A_{k+1})$ but not both. One can find a subsequence $(A_{k_j})$ of $(A_k)$ such that either
    \begin{enumerate}
        \item \label{asymptotic_x} $( \mathrm{pr}_1(A_{k_j}) )$ is a constant sequence and $(\mathrm{pr}_2(A_{k_j}))$ is a strictly increasing sequence, or
        \item \label{asymptotic_y} $(\mathrm{pr}_2(A_{k_j}))$ is a constant sequence and $(\mathrm{pr}_1(A_{k_j}))$ is a strictly increasing sequence.
    \end{enumerate}
    In the case of (\ref{asymptotic_x}), $c x_0 = \mathrm{pr}_1(A_{k_1})+1$ is a positive integer and the line $x=c x_0$ is the vertical asymptote of $P(c\vp)$. Similarly, if $(A_{k_j})$ satisfies (\ref{asymptotic_y}), then $c y_0 = \mathrm{pr}_2(A_{k_1}) + 1$ is a positive integer and the line $y=c y_0$ is the horizontal asymptote of $P(c\vp)$. Note that $x = cx_0$ in the case of (\ref{asymptotic_x}) and $y=cy_0$ in the case of (\ref{asymptotic_y}) do not intersect with $P(c\vp)$.

    Conversely, without loss of generality, we may assume that (1) in the statement holds. We want to show that $\frac{k}{x_0}$ is a cluster point of jumping numbers for every positive integer $k$. Since the line $x = k$ is the vertical asymptote of $P( \frac{k}{x_0} \vp)$ which does not meet $P(\frac{k}{x_0} \vp)$. One can take a sequence $(\epsilon_j)$ of positive real numbers and a sequence $(\alpha_j)$ of positive integers such that
    \begin{enumerate}
        \item $\epsilon_j > \epsilon_{j+1}$ for every $j \ge 1$ and $\epsilon_j \to 0$,
        \item $\alpha_j < \alpha_{j+1}$ for every $j \ge 1$,
        \item $B_j \in \interior P( ( \frac{k}{x_0} - \epsilon_j ) \vp )$ but $B_j \notin \interior P( (\frac{k}{x_0} - \epsilon_{j+1} ) \vp )$ for every $j \ge 1$,
    \end{enumerate}
    where $B_j = ( k , \alpha_j )$. This means that there is a sequence $(\lambda_j)$ of jumping numbers such that
    $$ \frac{k}{x_0} - \epsilon_j < \lambda_j \le \frac{k}{x_0} - \epsilon_{j+1} $$
    for each $j$. Since $\lambda_j$ converges to $\frac{k}{x_0}$, we conclude that $\frac{k}{x_0}$ is a cluster point of jumping numbers. The last assertion follows immediately from the above argument.
\end{proof}

    As a corollary of Theorem~\ref{cluster_asymptote}, we have

\begin{corollary} \label{dim2_cluster}
    In the setting of Theorem~\ref{mtimes} with $n=2$, the set of all the cluster points in $\jump(\vp)_0$ is discrete. In other words, there is no cluster point of cluster points of jumping numbers in this case.
\end{corollary}

 Now we can strengthen Theorem~\ref{guanli} of \cite{GL}  as follows.

\begin{corollary}\label{gl1}

 Let $\vp$ be the psh function   \eqref{glpsh} of \cite{GL} in Theorem~\ref{guanli}. The set $T$ of cluster points of $\jump(\vp)_0$ is precisely equal to the set of positive integers $\ZZ_{>0}$.

\end{corollary}

\begin{proof}

 From Theorem~\ref{mtimes} and Theorem~\ref{guanli}, we know that $T \supset \ZZ_{>0} $.
  The other inclusion can be shown by proving that $P(\vp)$ intersects with the line $y=0$ in view of Theorem~\ref{dim2_cluster} and $1$ is the smallest cluster point of jumping numbers. Let $\widehat{\vp}$ be the convex function associated to $\vp$. Then the gradient of $\widehat{\vp}$ at $(x,y)$ is
    $$ \nabla \widehat{\vp} (x,y) = \left( 1 + \sum_{k=1}^{\infty}{ \frac{ M^{-k} e^x}{ e^x + k^{-b_k} e^{b_k y} } }, \sum_{k=1}^{\infty}{ \frac{ M^k k^{-b_k} e^{b_ky}}{ e^x + k^{-b_k} e^{b_k y} } } \right) $$
    and this is contained in $P(\vp)$(see \cite[p.\,1015]{G}). Therefore, we have
    $$ \nabla \widehat{\vp} (0,y) = \left( 1 + \sum_{k=1}^{\infty}{ \frac{ M^{-k} }{ 1 +  k^{-b_k} e^{b_k y} } }, \sum_{k=1}^{\infty}{ \frac{M^k e^{b_k y}}{k^{b_k} + e^{b_k y} } } \right) \in P(\vp) $$
    and observe that $\nabla \widehat{\vp} (0,y)$ converges to a point $(\alpha,0)$, where $\alpha = 1 + \frac{1}{M-1}$, when $y$ tends to $-\infty$. Moreover, since $\alpha \le 2$ for $M \ge 2$ and $P(\vp)$ does not equal to $(\alpha,0) + \RR_{\ge 0}^2$, the point $(\alpha,2) \in \RR^2$ should be in the interior of $P(\vp)$. Therefore the point $(2,2)$ is contained in the interior of $P(\vp)$. We have $(1,1) \in \interior P(\frac{1}{2} \vp)$ and thus $1 \in \JJ (\frac{1}{2} \vp )$. We conclude that $\JJ(c \vp) = \OO_{\CC^2 , 0}$ whenever $0 \le c \le \frac{1}{2}$ and hence $1$ is the smallest cluster point of jumping numbers.
\end{proof}

 It is also easy to see that for this example \eqref{glpsh}, we have the case (e) and then necessarily (d) (since cluster points exist) of Figure 2.

\begin{remark1}
    In Example~\ref{vp+psi}, the vertical asymptote of $P(\vp + \psi)$ is $x=2$ and the horizontal asymptote of $P(\vp+\psi)$ is $y=1$. These asymptotes do not intersect with $P(\vp+\psi)$. By Theorem~\ref{cluster_asymptote}, the set of all cluster points of jumping numbers is $\{ \frac{k}{2} : k \in \ZZ_{>0} \}$. This can be also checked by direct computation using $C(\vp+\psi)$ (see (\ref{jump_newton}) ).
\end{remark1}

\section{Appendix  by S\'ebastien Boucksom}\footnote{S\'ebastien Boucksom :  CNRS–CMLS,
\'Ecole Polytechnique, e-mail: sebastien.boucksom@polytechnique.edu}

  Theorem~\ref{asymptotic} states that given a graded system of ideals $\mfab$,  its asymptotic multiplier ideals are equal to the analytic multiplier ideals of a Siu psh function $\vp = \vp_{\mfab}$ (on $U \subset X$) associated to $\mfab$.

 \begin{proof}[Proof of Theorem~\ref{asymptotic}]

  Before using the valuative characterizations of multiplier ideal sheaves, we will first prove that  $ v (\vp) = v (\mfab) $ for every divisorial valuation $v$ with nonempty center on $U \subset X$.

  We know that $v (\mfab) = \inf_k \frac{1}{k} v(\mfa_k) = \lim_k \frac{1}{k} v(\mfa_k)  $  from \cite[Lemma 2.3]{JM}. Now following \cite[p.258]{BEGZ}, let $\phi_k :=  \frac{1}{k!} \log \abs{\mfa_{k!}} $ (in the notation of \eqref{siu}). Since $v(\varphi)  \le  v (\phi_k) $ for every $k$, it follows that $v(\varphi) \le v(\mfab)$.

  On the other hand, by adding constants, we can arrange that $\vp$ is the increasing limit of $\phi_k$. Then we have  $v(\varphi) \ge \limsup_k v(\phi_k)  = v(\mfab) .$

  Now the theorem will follow from valuative characterization of multiplier ideals of both sides of $\JJ(c \vp) = \JJ(c \cdot \mfab)$, i.e.  \cite[Theorem 5.5]{BFJ} and     \cite[Theorem 4.1 (b)]{BFFU} respectively. The former states that $f \in \JJ(c \vp)$ if and only if there exists $\ep > 0$ such that $ v(f) \ge (1+\ep) v(c \vp) - A (v)  $ for all divisorial valuations with center on $U \subset X$.

  The latter states that $f \in \JJ(c \cdot \mfab)$ if and only if for every normalizing subscheme $N \subset X$ containing $0$ and every $0 < \delta \ll 1$,  $ v(f) \ge c v(\mfab) - A (v) + \delta v(\mathcal I_N) $ for all divisorial valuations with center on $U \subset X$.

  Since $\JJ (c \cdot \mfab) = \JJ (\frac{c}{m} \mfa_m)$ for some $m \ge 1$, we have $\JJ (c \cdot \mfab ) \subset \JJ(c \vp)$.  For the other direction of inclusion, suppose that there exists $\ep > 0$ such that $ v(f) \ge (1+\ep) v(c \vp) - A (v)  $ holds.  Choose $\delta$ sufficiently small so that $\delta v(\mathcal I_N) = \delta < \ep v (c \vp)$.  Then we have $$ v(f) \ge (1+\ep) v(c \vp) - A (v)  >   c v (\mfab) - A (v) + \delta v(\mathcal I_N)   $$  since $ v( c \vp) = c v( \vp) = c v (\mfab)$. This completes the proof.

 \end{proof}

\footnotesize

\bibliographystyle{amsplain}

\qa

\qa

\normalsize

\noi \textsc{Dano Kim}

\noi Department of Mathematical Sciences \& Research Institute of Mathematics

\noi Seoul National University, 08826  Seoul, Korea

\noi e-mail: kimdano@snu.ac.kr

\qa

\noi \textsc{Hoseob Seo}

\noi Department of Mathematical Sciences

\noi Seoul National University, 08826  Seoul, Korea

\noi e-mail: hskoot@snu.ac.kr

\end{document}